\documentclass[12pt]{amsart}
\usepackage{amssymb,latexsym, amsmath, amscd, array, graphicx}

\swapnumbers
\numberwithin{equation}{section}

\theoremstyle{plain}

\newtheorem{thm}{Theorem}[section]
\newtheorem{theorem}[thm]{Theorem}

\newtheorem{lemma}[thm]{Lemma}

\newtheorem{prop}[thm]{Proposition}
\newtheorem{cor}[thm]{Corollary}

\newcommand\theoref{Theorem~\ref}
\newcommand\lemref{Lemma~\ref}
\newcommand\propref{Proposition~\ref}
\newcommand\corref{Corollary~\ref}
\newcommand\defref{Definition~\ref}

\newcommand\Zp{\Z_p}

\theoremstyle{definition}

\newtheorem{definition}[thm]{Definition}

\newtheorem{rem}[thm]{Remark}

\newtheorem{ex}[thm]{Example}

\newtheorem{question}[thm]{Question}

\DeclareMathOperator{\sys}{{\rm sys}}

\DeclareMathOperator{\cat}{{\mbox{\rm cat$_{\rm LS}$}}}

\def\Im{\protect\operatorname{Im}}

\def\Hom{\protect\operatorname{Hom}}

\def\pt{\protect\operatorname{pt}}

\def\wgt{\protect\operatorname{wgt}}
\def\cwgt{\protect\operatorname{cwgt}}

\def\cd{\protect\operatorname{cd}}
\def\ber{\mathfrak b}

\newcommand \pa[2]{\frac{\partial #1}{\partial #2}}

\def\gf{\varphi}
\def\ga{\alpha}

\def\scr{\mathcal}
\def\A{{\scr A}}
\def\B{{\scr B}}
\def\O{{\scr O}}

\def\Z{{\mathbb Z}}

\def\R{{\mathbb R}}

\def\zp{\Z[\pi]}

\def\1{\hbox{\rm\rlap {1}\hskip.03in{\rom I}}}
\def\Bbbone{{\rm1\mathchoice{\kern-0.25em}{\kern-0.25em}
{\kern-0.2em}{\kern-0.2em}I}}

\def\pa{\partial}

\def\wt{\widetilde}

\def\m{\medskip}
\def\ov{\overline}

\long\def\forget#1\forgotten{} %

\newcommand\ver[1]{\marginpar{\tiny Changed in Ver \VER}}

\date{\today}
\begin{document}

\title[Small values of LS category for manifolds]{Small values of the
Lusternik-Schnirelmann category for manifolds}

\author[A.~Dranishnikov]{Alexander N. Dranishnikov$^{1}$} %
\thanks{$^{1}$Supported by NSF, grant DMS-0604494}

\author[M.~Katz]{Mikhail G. Katz$^{2}$} %
\thanks{$^{2}$Supported by the Israel Science Foundation (grants 84/03
and 1294/06) and the BSF (grant 2006393)}

\author[Yu.~Rudyak]{Yuli B. Rudyak$^{3}$}%
\thanks{$^{3}$Supported by NSF, grant 0406311}

\address{Alexander N. Dranishnikov, Department of Mathematics, University
of Florida, 358 Little Hall, Gainesville, FL 32611-8105, USA}
\email{dranish@math.ufl.edu}

\address{Mikhail G. Katz, Department of Mathematics, Bar Ilan University,
Ramat Gan 52900 Israel}
\email{katzmik ``at'' math.biu.ac.il}

\address{Yuli B. Rudyak, Department of Mathematics, University
of Florida, 358 Little Hall, Gainesville, FL 32611-8105, USA}
\email{rudyak@math.ufl.edu}



\begin{abstract}
We prove that manifolds of Lusternik-Schnirelmann category~$2$
necessarily have free fundamental group.  We thus settle a 1992
conjecture of Gomez-Larra\~naga and Gonzalez-Acu\~na, by generalizing
their result in dimension~$3$, to all higher dimensions.  We also
obtain some general results on the relations between the fundamental
group of a closed manifold~$M$, the dimension of~$M$, and the
Lusternik-Schnirelmann category of~$M$, and relate the latter to the
systolic category of~$M$.
\end{abstract}

\maketitle
\tableofcontents

\section{Introduction}

We follow the normalization of the Lusternik-Schnirelmann category (LS
category) used in the recent monograph \cite{CLOT} (see
Section~\ref{three} for a definition).  Spaces of LS category~$0$ are
contractible, while a closed manifold of LS category~$1$ is homotopy
equivalent (and hence homeomorphic) to a sphere.

The characterization of closed manifolds of LS category~$2$ was
initiated in 1992 by J.~Gomez-Larra\~naga and F.~Gonzalez-Acu\~na
\cite{GoGo} (see also~\cite{OR}), who proved the following result on
closed manifolds~$M$ of dimension~$3$: the fundamental group of~$M$ is
free and nontrivial if and only if its LS category is~$2$.
Furthermore, they conjectured that the fundamental group of every
closed~$n$-manifold,~$n\ge 3$, of LS category~$2$ is necessarily
free~\cite[Remark, p.~797]{GoGo}.  Our interest in this natural
problem was also stimulated in part by recent work on the comparison
of the LS category and the systolic category~\cite{KR1, KR2, SGT},
which was inspired, in turn, by M.~Gromov's systolic inequalities
\cite{Gr1, Gr2, Gr3, Gr4}.

In the present text we prove this 1992 conjecture.  Recall that all
closed surfaces different from~$S^2$ are of LS category 2.

\begin{theorem}
\label{11}
A closed connected manifold of LS category~$2$ either is a surface, or
has free fundamental group.
\end{theorem}

\begin{cor}
\label{12b}
Every manifold~$M^n, n\geq 3$, with non-free fundamental group
satisfies~$\cat(M)\geq 3$.
\end{cor}

We found that there is no restriction on the fundamental group
for closed manifolds of LS category 3. In particular we proved
the following.
\begin{theorem}
\label{t:realization}
Given a finitely presented group~$\pi$ and non-negative integers~$k,
l$, there exists a closed manifold~$M$ such that~$\pi_1(M)=\pi$,
while~$\cat M=3+k$ and~$\dim M=5+2k+l$.  Furthermore, if~$\pi$ is not
free, then~$M$ can be chosen~$4$-dimensional with~$\cat M=3$.
\end{theorem}

Thus, there is no restriction on the fundamental group of manifolds of
LS category 3 and higher.

The above results lead to the following questions:

\begin{question}
\label{4-problem}
If a 4-dimensional CW-complex ~$X$ has free fundamental group, then we
have the bound~$\cat X \leq 3$.  Is the stronger bound~$\cat X \leq 2$
necessarily satisfied?
\end{question}

We prove the inequality $\cat M \le n-2$ for connected~$n$-manifolds
with free fundamental group and~$n>4$, see \propref{p:n>4}.
In~\cite{S2}, J.~Strom proved a stronger inequality~$\cat X\le
\frac{2}{3}\dim X$ for an arbitrary CW-space $X$.  Later, it was
proved in~\cite{Dr} that if the fundamental group is free, then the
bound
\begin{equation}
\cat X\le \frac{1}{2}\dim X +1,
\end{equation}
is satisfied by every CW-complex~$X$.

The above Question~\ref{4-problem} has an affirmative answer when~$M$
is a closed orientable manifold, in view of a theorem due to
J.~A.~Hillman~\cite{Hi} which states that a closed~$4$-dimensional
manifold with free fundamental group has a CW-decomposition in which
the three-skeleton has the homotopy type of a wedge of spheres.

\begin{question} 
Is it true that~$\cat (M \setminus \{\pt\})=1$ for any closed
manifold~$M$ with~$\cat M=2$?  This is proved in~\cite{GoGo} for the
case~$\dim M=3$. A direct proof would imply the main theorem
trivially.
\end{question}

\begin{question}
Given integers~$m$ and~$n$, describe the fundamental groups of closed
manifolds~$M$ with~$\dim M=n$ and~$\cat M=m$.
\end{question}

Note that in the case~$m=n$, the fundamental group of~$M$ is of
cohomological dimension~$\ge n$, see e.g. the Berstein--\v Svarc
\theoref{t:berstein}.
Thus, we can ask when the converse holds.

\begin{question}
Given a finitely presented group~$\pi$ and an integer $n\geq 4$ such
that~$H^n(\pi)\not = 0$, when can one find a closed manifold~$M$
satisfying~$\pi_1(M)=\pi$ and $\dim M=\cat M=n?$ Note that
\propref{p:lense} shows that such a manifold~$M$ does not always
exist.
\end{question}

A related numerical invariant called the {\em systolic category\/} can
be thought of as a Riemannian analogue of the LS category~\cite{SGT}.
In \cite{DKR2} we apply \corref{12b} to prove that the systolic
category of a~$4$-manifold is a lower bound for its LS category.

\begin{theorem}
\label{12}
Every closed orientable~$4$-manifold~$M$ satisfies the
inequality~${\rm cat}_{\sys}(M) \leq \cat(M)$.
\end{theorem}

In particular, this inequality implies that if a 4-manifold $M$ has a
free fundamental group then ${\rm cat}_{\sys}(M) = \cat(M)$.  In a
related development in systolic topology, an intriguing model for
$BS^3$ built out of $BS^1$ was used in \cite{e7} to prove that the
symmetric metric of the quaternionic projective space, contrary to
expectation, is {\em not\/} its systolically optimal metric.

\m The proof of the main theorem proceeds roughly as follows. If the
group~$\pi:=\pi_1(M)$ is not free, then by a result of J.~Stallings
and R.~Swan, the group~$\pi$ is of cohomological dimension at
least~$2$.  We then show that~$\pi$ carries a suitable
nontrivial~$2$-dimensional cohomology class~$u$ with twisted
coefficients, and of category weight~$2$.  Viewing~$M$ as a subspace
of~$K(\pi,1)$ that contains the~$2$-skeleton~$K(\pi,1)^{(2)}$, and
keeping in mind the fact that the~$2$-skeleton carries the fundamental
group, we conclude that the restriction (pullback) of~$u$ to~$M$ is
non-zero and also has category weight~$2$.  By Poincar\'e duality with
twisted coefficients, one can find a complementary~$(n-2)$-dimensional
cohomology class.  By a category weight version of the cuplength
argument, we therefore obtain a lower bound of~$3$ for~$\cat M$.

In Section~\ref{local}, we review the material on local coefficient
systems, a twisted version of Poincar\'e duality, and 2-dimensional
cohomology of non-free groups.  In Section~\ref{three}, we review the
notion of category weight.  In Section~\ref{four}, we prove our main
result, Theorem~\ref{11}.  In Section~\ref{higher} we prove
\theoref{t:realization}.

\section{Cohomology with local coefficients}
\label{local}

A {\it local coefficient system}~$\A$ on a path
connected~CW-space~$X$ is a functor from the fundamental
groupoid~$\Gamma(X)$ of~$X$, to the category of abelian groups.  See
\cite{Ha}, \cite{Wh} for the definition and properties of local
coefficient systems.

In other words, an abelian group~$\A_x$ is assigned to each
point~$x\in X$, and for each path~$\ga$ joining~$x$ to~$y$, an
isomorphism~$\ga^*:\A_y \to \A_x$ is given.  Furthermore, paths that
are homotopic are required to yield the same isomorphism.

Let~$\pi=\pi_1(X)$, and let~$\zp$ be the group ring of~$\pi$.  Note
that all the groups~$A_x$ are isomorphic to a fixed group~$A$.  We
will refer to~$A$ as a {\it stalk\/} of~$\A$.

Given a map~$f: Y \to X$ and a local coefficient system~$\A$ on~$X$,
we define a local coefficient system on~$Y$, denoted~$f^*\A$, as
follows.  The map~$f$ yields a functor~$\Gamma(f): \Gamma(Y) \to
\Gamma(X)$, and we define~$f^*\A$ to be the functor~$\A\circ
\Gamma(f)$. Given a pair of coefficient systems~$\A$ and~$\B$, the
tensor product~$\A\otimes \B$ is defined by setting~$(\A\otimes
\B)_x=\A_x\otimes\B_x$.

\m \begin{ex}\label{ex:bundle} A useful example of a local coefficient
system is given by the following construction.  Given a fiber bundle
$p: E \to X$ over~$X$, set~$F_x=p^{-1}(x)$.  Then the family
$\{H_k(F_x)\}$ can be regarded a local coefficient system, see
\cite[Example 3, Ch. VI, \S 1]{Wh}.  An important special case is that
of an~$n$-manifold~$M$ and spherical tangent bundle~$p: E \to M$ with
fiber~$S^{n-1}$, yielding a local coefficient system~$\O$ with
$\O_x=H_{n-1}(S^{n-1}_x)\cong \Z$.  This local system is called the
{\it orientation sheaf\/} of~$M$.
\end{ex}

\begin{rem}
\label{22}
There is a bijection between local coefficients on~$X$ and
$\zp$-modules \cite[Ch. 1, Exercises F]{Sp}. If~$\A$ is a local
coefficient system with stalk~$A$, then the natural action of the
fundamental group on~$A$ turns~$A$ into a~$\Z[\pi]$-module.
Conversely, given a~$\Z[\pi]$-module~$A$, one can construct a local
coefficient system~$\scr L(A)$ such that induced~$\Z[\pi]$-module
structure on~$A$ coincides with the given one, cf.~\cite{Ha}.
\end{rem}

We recall the definition of the (co)homology groups with local
coefficients via modules \cite{Ha}:
\begin{equation}\label{cohomology}
H^k(X;\A)\cong H^k(\Hom_{\zp}(C_*(\wt X), A), \delta)
\end{equation}
and
\begin{equation}\label{homology}
H_k(X;\A)\cong H_k(A\otimes _{\zp}C_*(\wt X), 1\otimes \pa).
\end{equation}
Here~$(C_*(\wt X), \pa)$ is the chain complex of the universal
cover~$\wt X$ of~$X$,~$A$ is the stalk of the local coefficient
system~$\A$, and~$\delta$ is the coboundary operator.  Note that in
the tensor product we used the right~$\zp$ module structure on~$A$
defined via the standard rule~$ag=g^{-1}a$, for~$a\in A, g\in \pi$.

Recall that for~CW-complexes~$X$, there is a natural bijection
between equivalence classes of local coefficient systems and locally
constant sheaves on~$X$.  One can therefore define (co)homology with
local coefficients as the corresponding sheaf cohomology \cite{Br}.
In particular, we refer to \cite{Br} for the definition of the cup
product
\[
\cup: H^i(X;\A) \otimes H^j(X;\B) \to H^{i+j}(X; \A\otimes \B)
\]
and the cap product
\[
\cap: H_i(X;\A) \otimes H^j(X;\B) \to H_{i-j}(X; \A\otimes \B).
\]
A nice exposition of the cup  and the cap products in a slightly
different setting can be found in \cite{Bro}.  In particular, we have
the cap product
\[
H_k(X;\A)\otimes H^k(X;\B)\to H_0(X;\A\otimes \B)\cong
A\otimes_{\zp}B.
\]

\begin{prop}
\label{l:evaluation}
Given an integer~$k\ge 0$, there exists a local coefficient
system~$\B$ and a class~$v\in H^k(X;\B)$ such that, for every local
coefficient system~$\A$ and nonzero class~$a\in H_k(X;\A)$, we
have~$a\cap v\ne 0$.
\end{prop}

\begin{proof}
Throughout the proof $\otimes$ denotes $\otimes_{\zp}$. We convert the
stalk of~$\A$ into a right~$\zp$-module~$A$ as above.  Below we use the
isomorphisms~\eqref{cohomology} and \eqref{homology}.  Consider the
chain~$\zp$-complex
\[
\CD \dots @>>> C_{k+1}(\wt X) @>\pa_{k+1}>>C_k(\wt X) @>\pa_k>>
C_{k-1}(\wt X) @>>> \dots .
\endCD
\]
For the given~$k$, we set~$B: =C_k(\wt X)/\Im \pa_{k+1}$.  Let~$\B$ be
the corresponding local system on~$X$.  Thus, we obtain the exact
sequence of~$\zp$-modules
\[
\CD C_{k+1}(\wt X) @>\pa_{k+1}>>C_k(\wt X)@>f>>B\to 0.  \endCD
\]
Note that the epimorphism~$f$ can be regarded as a~$k$-cocycle with
values in~$\B$, since~$\delta f (x)=f\pa_{k+1}(x)=0$.  Let~$v:=[f]\in
H^k(X;\B)$ be the cohomology class of~$f$.  Now we prove that
$$
a\cap [f] \not=0.
$$

Since the tensor product is right exact, we obtain the diagram
\[
\CD A\otimes C_{k+1}(\wt X)
@>1\otimes\pa_{k+1}>>A\otimes C_k(\wt X)@>1\otimes
f>>A\otimes B @>>> 0\\ @. @. g @ VVV @.\\
@. @. A\otimes C_{k-1}(\wt X) \endCD
\]
where the row is exact.  The composition
\[
\CD A\otimes C_k(\wt X)@>1\otimes f>>A\otimes B @>g>>
A\otimes C_{k-1}(\wt X)
\endCD
\]
coincides with~$1\otimes \pa_k$. We represent the class~$a$ by a cycle
\[
z\in A\otimes  C_k(\wt X).
\]
Since~$z\notin \Im(1\otimes \pa_{k+1})$, we conclude that
$$
(1\otimes f)(z)\ne 0\in A\otimes B=H_0(X;\A\otimes \B).
$$
Thus, for the cohomology class~$v$ of~$f$ we have~$a\cap v \ne 0$.
\end{proof}

Every closed connected~$n$-manifold~$M$ satisfies~$H_n(M; \O)\cong
\Z$.  A generator (one of two) of this group is called the {\it
fundamental class} of~$M$ and is denoted by~$[M]$.

\m One has the following generalization of the Poincar\'e duality
isomorphism.

\begin{thm}[{\cite[Corollary 10.2]{Br}}]
\label{th:PD}
The homomorphism
\begin{equation}\label{eq:PD}
\Delta: H^i(M;\A)\to H_{n-i}(M;\O\otimes \A)
\end{equation}
defined by setting~$\Delta(a)=[M]\cap a$, is an isomorphism.
\end{thm}

In fact, in \cite{Br} there is the sheaf~$\O^{-1}$ at the right, but
for manifolds we have~$\O=\O^{-1}$.

Given a group~$\pi$ and a~$\Z[\pi]$-module~$A$, we denote by
$H^*(\pi;A)$ the cohomology of the group~$\pi$ with coefficients
in~$A$, see e.g.~\cite{Bro}.  Recall that~$H^i(\pi;A)=
H^i(K(\pi,1);\scr L(A))$, see Remark~\ref{22}.

\m Let~$\cd(\pi)$ denote the cohomological dimension of~$\pi$
over~$\Z$, i.e. the largest~$m$ such that there exists
an~$\Z[\pi]$-module~$A$ with~$H^m(\pi; A)\ne 0$.

\begin{thm}[\cite{Stal, Swan}]\label{th:free}
If~$\cd\pi\le 1$ then~$\pi$ is a free group.
\end{thm}

We will need the following known fact from the cohomology theory of
groups.

\begin{lemma}\label{l:cd}
If~$\pi$ be a group with~$\cd\pi=q\geq 2$.  Then~$H^2(\pi;A)\ne
0$ for some~$\zp$-module~$A$.
\end{lemma}

\begin{proof} 
We use the fact that cohomology of the group $\pi$ with coefficients
in an injective $\zp$-module are trivial and the fact that every
$\zp$-module $A'$ can be imbedded into an injective $\zp$module
$J$,~\cite{Bro}.  Let~$0\to A'\to J\to A''\to 0$ be an exact sequence
of~$\zp$-modules with~$J$ injective. Then by the coefficients long
exact sequence~$H^k(\pi;A')=H^{k-1}(\pi;A'')$ for $k>1$.
Since~$H^q(\pi;B)\ne 0$ for some~$B$, the proof can be completed by an
obvious induction.
\end{proof}

\section{Category weight and lower bounds for~$\cat$}
\label{three}

In this section, we review the notion of category weight and its
relation to the Lusternik-Schnirelmann category.

\begin{definition}[\cite{BG,Fe,F}]
\label{def:cat-map}
Let~$f: X \to Y$ be a map of (locally contractible)~CW-spaces. The
\emph{Lusternik--Schnirelmann category of~$f$}, denoted~$\cat(f)$, is
defined to be the minimal integer~$k$ such that there exists an open
covering~$\{U_0, \ldots, U_k\}$ of~$X$ with the property that each of
the restrictions~$f|A_i\colon A_i \to Y$,~$i=0,1, \ldots, k$ is
null-homotopic.

The \emph{Lusternik--Schnirelmann category~$\cat X$ of a space~$X$} is
defined as the category~$\cat (1_X)$ of the identity map.
\end{definition}

\begin{definition}\label{def:swgt}
The \emph{category weight}~$\wgt(u)$ of a non-zero cohomology class~$u \in
H^*(X; \A)$ is defined as follows:
\begin{equation*}
\label{31}
\wgt(u)\ge k \Longleftrightarrow \{\gf^*(u)=0 {\rm\ for\ every\ }
\gf\colon F \to X {\rm\ with\ } \cat(\gf) < k\}.
\end{equation*}
\end{definition}

\begin{rem}\label{rem:credits}\rm
E.~Fadell and S.~Husseini (see \cite{FH}) originally proposed the
notion of category weight.  In fact, they considered an invariant
similar to the~$\wgt$ of \eqref{31} (denoted in \cite{FH} by~$\cwgt$),
but where the defining maps~$\gf\colon F \to X$ were required to be
inclusions rather than general maps.  As a consequence,~$\cwgt$ is not
a homotopy invariant, and thus a delicate quantity in homotopy
calculations.  Yu.~Rudyak \cite{R1, R2} and J.~Strom \cite{S1}
proposed a homotopy invariant version of category weight as defined in
\defref{def:swgt}.
\end{rem}

\begin{prop}[\cite{R1,S1}]
\label{prop:swgtprops}
Category weight has the following properties.
\begin{enumerate}
\item~$1\le \wgt(u) \leq \cat(X)$, for all~$u \in \widetilde
H^*(X;\A), u\ne 0$.
\vskip3pt
\item For every~$f\colon Y \to X$ and~$u\in H^*(X;\A)$ with
$f^*(u)\not = 0$ we have
$\cat(f) \geq \wgt(u)$ and~$\wgt(f^*(u)) \geq \wgt(u)$.
\vskip3pt
\item For~$u\in H^*(X;\A)$ and~$v\in H^*(X;\B)$ we have
\[
\wgt(u\cup v) \geq \wgt(u) + \wgt(v).
\]  \vskip3pt
\item For every~$u \in H^s(K(\pi,1);\A)$,~$u\ne 0$, we have
$\wgt(u)\geq s$.  \vskip3pt
\end{enumerate}
\end{prop}

\begin{proof}
See \cite[\S 2.7 and Proposition 8.22]{CLOT}, the proofs in
loc. cit. can be easily adapted to local coefficient systems.
\end{proof}

\section{Manifolds of LS category~$2$}
\label{four}

In this section we prove that the fundamental group of a closed
connected manifold of LS category~$2$ is free.

\begin{thm}
\label{th:main}
Let~$M$ be a closed connected manifold of dimension at least~$3$.  If
the group~$\pi:=\pi_1(M)$ is not free, then~$\cat M \geq 3$.
\end{thm}

\begin{proof}
By \theoref{th:free} and \lemref{l:cd}, there a local coefficient
system~$\A$ on~$K(\pi,1)$ such that~$H^2(K(\pi,1);\A)\ne 0$.  Choose a
non-zero element~$u\in H^2(K(\pi,1);\A)$. Let~$f: M \to K(\pi,1)$ be
the map that induces an isomorphism of fundamental groups, and let~$i:
K\to M$ be the inclusion of the~$2$-skeleton. (If~$M$ is not
triangulable, we take~$i$ to be any map of a~$2$-polyhedron that
induces an isomorphism of fundamental groups.)  Then
\[
(fi)^*: H^2(K(\pi,1);\A) \to H^2(K;(fi)^*\A)
\]
is a monomorphism.  In particular, we have~$f^*u\not=0$
in~$H^2(M;(f)^*\A)$.  Now consider the class
\[
a=[M]\cap f^*u\in H_{n-2}(M;\O^{-1}\otimes f^*\A),
\]
where $n=\dim M$.  Then~$a\ne 0$ by \theoref{th:PD}.  Hence, by
Proposition~\ref{l:evaluation}, there exists a class~$v\in
H^{n-2}(M;\B)$ such that~$a\cap v \ne 0$. We claim that~$f^*u\cup v
\ne 0$. Indeed, one has
\[
[M]\cap (f^*u\cup v)=([M]\cap f^*u)\cap v =a\cap v\ne 0.
\]
Now,~$\wgt f^*u\ge 2$ by \propref{prop:swgtprops}, items (2)
\and~(4). Furthermore,~$\wgt(v)\geq 1$ by \propref{prop:swgtprops},
item~(1).  We therefore obtain the lower bound~$\wgt(f^*u\cup v)\ge 3$
by \propref{prop:swgtprops}, item (3).  Since~$f^*u\cup v\ne 0$, we
conclude that~$\cat M \ge 3$ by \propref{prop:swgtprops}, item (1).
\end{proof}

\begin{cor}
If~$M^n, n\geq 3$ is a closed manifold with~$\cat M\le 2$, then
$\pi_1(M)$ is a free group.
\end{cor}

\begin{rem}
An alternative approach to Theorem~\ref{th:main} would be using the
Berstein-\v{S}varc class $\ber\in H^1(\pi;I(\pi))$ where $I(\pi)$ is
the augmentation ideal of $\pi$. If~$ \cd(\pi)\ge 2$ then $\ber^{2}
\ne 0$ by \cite{DR} (see also \theoref{t:berstein}).  In
particular,~$H^2(\pi;I(\pi)\otimes I(\pi))\ne 0$, and we obtain an
alternative proof of \lemref{l:cd}.
\end{rem}

The following Proposition is a special case of ~\cite[Corollary
4.2]{Dr}.  Here we give a relatively simple geometric proof.

\begin{prop}\label{p:n>4}
Let $M$ be a closed connected $n$-dimensional PL manifold, $n>4$, with
free fundamental group.  Then $\cat M \le n-2$.
\end{prop}

\begin{proof}
If $X$ is a 2-dimensional (connected) CW-complex with free fundamental
group then $\cat X\le 1$, see e.g.~\cite[Theorem 12.1]{KRS}.  Hence,
if $Y$ is a $k$-dimensional complex with free fundamental group then
$\cat Y\le k-1$ for $k>2$.  Now, let $K$ be a triangulation of $M$,
and let $L$ be its dual triangulation.  Then $M\setminus L^{(l)}$ is
homotopy equivalent to $K^{(k)}$ whenever $k+l+1=n$.  Hence,
\[
\cat M\le \cat K^{(k)}+\cat L^{(l)}+1.
\]
Since $\pi_1(K)$ and $\pi_1(L)$ are free, we conclude that $\cat
K^{(k)}\le k-1$ and $\cat L^l\le l-1$ for $k,l>1$. Thus $\cat M\le k-1
+l-1 +1=n-2$.
\end{proof}

\section{Manifolds of higher LS category}
\label{higher}

Gromov~\cite[4.40]{Gr3} called a polyhedron $X$ {\em $n$-essential} if
there is no map $f:X\to K(\pi,1)^{(n-1)}$ to the $(n-1)$-dimensional
skeleton of an Eilenberg-MacLane complex that induces an isomorphism
of the fundamental groups. We extend his definition as follows.

\begin{definition}
\label{def:k-essent}
A~CW-space~$X$ is called {\em strictly $k$-essential},~$k>1$ if for
every CW-complex structure on~$X$ there is no map between the skeleta
$f:X^{(k)}\to K(\pi,1)^{(k-1)}$ that induces an isomorphism of the
fundamental groups.
\end{definition}

Clearly, a strictly $n$-essential space is Gromov $n$-essential, while
the converse is false. Furthermore, an $n$-dimensional polyhedron is
strictly $n$-essential if it is Gromov $n$-essential.

\begin{thm}\label{th:k-essent}
Let $M$ be a closed strictly ~$k$-essential manifold.  If its
dimension satisfies~$\dim M \geq k+1$, then its LS category also
satisfies $\cat M\ge k+1$.
\end{thm}

\begin{proof}
We first consider the case~$k=2$. If~$\cat M\le 2$, then, by
\theoref{th:main},~$\pi_1(M)$ is free.  Hence there is a map~$f:M\to
\vee S^1$ that induces an isomorphism of the fundamental groups,
and~$M$ is not strictly $2$-essential.

Now assume~$k\ge 3$.  Let~$K=K(\pi_1(M),1)$.  Consider a map
\[
f: M^{(k-1)} \to K^{(k-1)}
\]
such that the restriction~$f|_{M^{(2)}}$ is the identity homeomorphism
of the 2-skeleta~$M^{(2)}$ and~$K^{(2)}$. We consider the problem of
extension of~$f$ to~$M$.

We claim that the first obstruction ~$o(f)\in H^k(M;E)$ (taken with
coefficients in a local system~$E$ with the stalk
$\pi_{k-1}(K^{(k-1)})$) to the extension is not equal to zero.

Indeed, if~$o(f)=0$, then there exists a map~$\ov f:M^{(k)}\to
K^{(k-1)}$ which coincides with~$f$ on the~$(k-2)$-skeleton. The
map
\[
\ov f_*: \pi_1(M^{(k)})\to \pi_1(K^{(k-1)})
\]
can be viewed as an endomorphism of~$\pi_1(M)$ that is identical on
generators, and therefore~$\ov f_*$ is an isomorphism.  Hence~$M$ is
not strictly~$k$-essential.

Consider the commutative diagram
\[
\CD
M^{(k-1)} @>f>> K^{(k-1)} @>{\rm id}>> K^{(k-1)}\\
@ViVV @VjVV @.\\
M @>\tilde f>> K
\endCD
\]
where~$i$ and~$j$ are the inclusions of the skeleta.  Let~$\alpha$ be
the first obstruction to the extension of id to a map~$K \to
K^{(k-1)}$.  By commutativity of the above diagram, we
have~$o(f)=\tilde f^*(\alpha)$.  Now, asserting as in the proof of
\theoref{th:main}, we get that~$\tilde f^*(\alpha)\cup v\ne 0$ for
some~$v$ with~$\dim v=\dim M-k$.  Since~$\dim M>k$, we conclude
that~$\dim v\ge 1$ and thus~$\cat M\ge k+1$.
\end{proof}

\begin{rem}
If a closed manifold~$M^n$ is~$n$-essential then~$\cat M=n$, see
e.g.~\cite{KR1} and~\cite[Theorem 12.5.2]{SGT}.
\end{rem}

The following theorem for $n\geq 3$ was proven in~\cite[Theorem
A]{Ber} and~\cite[Theorem 20]{Sv}, see also~\cite[Proposition
2.51]{CLOT}. The case $n=2$ was proved in \cite{DR}.

\begin{theorem}
\label{t:berstein}
If $\dim X=\cat X=n$, then $u^{n}_X\ne 0$ where
$u_X=j^*(\ber)\in H^1(X;I(\pi))$, $j:X\to K(\pi,1)$ induces an isomorphism
of the fundamental groups, and $\ber\in H^1(\pi, I(\pi))$ is the
Berstein-\v{S}varc class.
 {\rm (}For the case
~$n=\infty$ this means that~$u^{k}\ne 0$ for all~$k$.{\rm )}
\end{theorem}

\begin{prop}\label{p:nonfree}
For every non-free finitely presented group~$\pi$, there exists a
closed~$4$-dimensional manifold~$M$ with fundamental group~$\pi$
and~$\cat M=3$.
\end{prop}

\begin{proof}
Let $K$ be a~$2$-skeleton of~$K(\pi,1)$. Take an embedding of $K$ in~$\R^5$
and let~$M=\partial N$ be the boundary of the regular neighborhood
$N$ of this skeleton. Then there is a retraction $N \to K$, and, clearly,
the map $f: M \subset N \to K$ induces an isomorphism of fundamental groups.
Now, let $u_M\in H^1(M; I(\pi))$ be the class described in
the Theorem~\ref{t:berstein}.
Then $u_M=f^*u_K$, and hence $u_M^{4}=0$.
 Therefore~$\cat M<4$ by \theoref{t:berstein}, and thus $\cat M= 3$.
\end{proof}

Let~$M_f$ be the mapping cylinder of~$f:X\to Y$.  We use the notation
$\pi_*(f)=\pi_*(M_f,X)$.  Then~$\pi_i(f)=0$ for~$i\le n$ amounts to
saying that it induces isomorphisms~$f_*:\pi_i(X_1)\to \pi_i(Y_1)$ for
$i\le n$ and an epimorphism in dimension~$n+1$. Similar notation
$H_*(f)=H_*(M_f,X)$ we use for homology.

\begin{lemma}
\label{l:join} Let~$f_j:X_j\to Y_j$ be a family of maps of~CW-spaces such
that~$H_i(f_j)=0$ for~$i\le n_j$. Then~$ H_i(f_1\wedge
\cdots \wedge f_s)=0$ for~$i\le \min\{n_j\}$.
\end{lemma}

\begin{proof}
Note that
\[
M(f_1\wedge \cdots \wedge f_s)\cong Y_1\wedge \cdots \wedge Y_s\cong
M(f_1)\wedge \cdots \wedge M(f_s).
\]
Now, by using the K\"unneth formula and considering the homology exact
sequence of the pair~$(M(f_1)\wedge\cdots \wedge M(f_s), X_1\wedge
\cdots \wedge X_s)$, we obtain the result.
\end{proof}

\begin{prop}\label{p:join}
Let~$f_j:X_j\to Y_j$,~$3\le j\le s$ be a family of maps of~CW-spaces such
that~$\pi_i(f_j)=0$ for~$i\le n_j$.  Then the joins
satisfy
\[
\pi_k(f_1\ast f_2\ast\dots\ast f_s)=0
\]
for~$k\le\min\{n_j\}+s-1$.
\end{prop}

\begin{proof}
By the version of the Relative Hurewicz Theorem for non-simply
connected~$X_j$ \cite[Theorem 4.37]{Ha}, we obtain~$H_i(f_j)=0$ for
$i\le n_j$.  By \lemref{l:join} we obtain
that~$H_k(f_1\wedge\cdots\wedge f_s)=0$ for~$k\le\min\{n_j\}$.  Since
the join~$A_1\ast\dots\ast A_s$ is homotopy equivalent to the iterated
suspension~$\Sigma^{s-1}(A_1\wedge\dots\wedge A_s)$ over the smash
product, we conclude that~$H_k(f_1\ast\dots\ast f_s)=0$
for~$k\le\min\{n_j\}+s-1$.  Since~$X_1\ast\dots\ast X_s$ is simply
connected for~$s\ge 3$, by the standard Relative Hurewicz Theorem we
obtain that~$\pi_k(f_1\ast\dots \ast f_s)=0$ for $k\le
\min\{n_j\}+s-1$.
\end{proof}

Given two maps~$f:Y_1\to X$ and~$g:Y_2\to X$, we set
\[
Z=\{(y_1,y_2,t)\in Y_1\ast Y_2\mid f(y_1)=g(y_2)\}
\]
and define the {\em fiberwise join}, or {\em join over~$X$} of~$f$
and~$g$ as the map
\[
f{\ast_X}g:Z\to X,\quad (f{\ast_X}g)(y_1,y_2, t)=f(y_1)
\]
Let~$p_0^X:PX\to X$ be the Serre path fibration. This means that
$PX$ is the space of paths on~$X$ that start at the base point of the
pointed space~$X$, and~$p_0(\alpha)=\alpha(1)$.  We denote by
$p_n^X;G_n(X)\to X$ the~$n$-fold fiberwise join of~$p_0$.

The proof of the following theorem can be found in \cite{CLOT}.

\begin{thm}[Ganea, \v Svarc]
\label{t:ganea}
For a~CW-space~$X$,~$\cat(X)\le n$ if and only if there exists a
section of~$p_n:G_n(X)\to X$.
\end{thm}

\begin{prop}\label{p:sum}
The connected sum~$S^k\times S^l\#\cdots \#S^k\times S^l$ is a space
of LS-category~$2$.
\end{prop}

\begin{proof}
This can be deduced from a general result of K. Hardy~\cite{H} because
the connected sum of two manifolds can be regarded as the double
mapping cylinder. Alternatively, one can note that, after removing a
point, the manifold on hand is homotopy equivalent to the wedge of
spheres.
\end{proof}

\begin{thm}
\label{t:high}
For every finitely presented group~$\pi$ and~$n\ge 5$, there is a
closed~$n$-manifold~$M$ of LS-category~$3$ with~$\pi_1(M)=\pi$.
\end{thm}

\begin{proof}
If the group~$\pi$ is the free group of rank~$s$, we let~$M'$ be the
$k$-fold connected sum~$S^1\times S^{2} \# \cdots \#S^1\times S^{2}$.
Then~$M'$ is a closed~$3$-manifold of LS category~$2$ with
$\pi_1(M')=F_s$.  Then the product manifold~$M=M'\times S^{n-3}$ has
cuplength~$3$ and is therefore the desired manifold.

Now assume that the group~$\pi$ is not free.  We fix a presentation
of~$\pi$ with~$s$ generators and~$r$ relators.  Let~$M'$ be
the~$k$-fold connected sum~$S^1\times S^{n-1} \# \cdots \#S^1\times
S^{n-1}$.  Then~$M'$ is a closed~$n$-manifold of the category~$2$ with
$\pi_1(M')=F_s$.  For every relator~$w$ we fix a nicely imbedded
circle~$S^1_w\subset M'$ such that~$S_w^{-1}\cap S_v^{-1}=\emptyset$
for~$w\ne v$.  Then we perform the surgery on these circles to obtain
a manifold~$M$. Clearly,~$\pi_1(M)=\pi$. We show that~$\cat(M)\le 3$,
and so~$\cat M =3$ by \theoref{th:main}.

As usual, the surgery process yields an~$(n+1)$-manifold~$X$ with
$\partial X=M\sqcup M'$.  Here~$X$ is the space obtained from
$M'\times I$ by attaching handles~$D^2\times D^{n-1}$ of index 2 to
$M'\times 1$ along the above circles. We note that~$\cat(X)\le 3$.

On the other hand, by duality,~$X$ can be obtained from~$M \times I$
by attaching handles of index~$n-1$ to the boundary component of~$M
\times I$.  In particular, the inclusion~$f: M \to X$ induces an
isomorphism of the homotopy groups of dimension~$\le n-3$ and an
epimorphism in dimension~$n-2$.  Hence the map
\[
\Omega f:\Omega M\to\Omega X
\]
induces isomorphisms in dimensions~$\le n-4$ and an epimorphism in
dimension~$n-3$.  Thus,~$\pi_i(\Omega f)=0$ for~$i\le n-3$.

In order to prove the bound~$\cat M\le 3$, it suffices to show that
the Ganea-\v Svarc fibration~$p_3:G_3(M)\to M$ has a section.
Consider the commutative diagram
\[
\CD
G_3M @>q> >Z @>f'>> G_3(X)\\
@Vp_M^3VV @Vp'VV @ VVp_3^XV\\
M @= M @>f>> X\\
\endCD
\]
where the right-hand square is the pull-back diagram
and~$f'q=G_3(f)$. Note that~$q$ is uniquely determined.
Since~$\cat(X)\le 3$, by \theoref{t:ganea} there is a section~$s:X\to
G_3(X)$.  It defines a section~$s':M\to Z$ of~$p'$. It suffices to
show that the map~$s':M\to Z$ admits a homotopy lifting~$h: M \to
G_3M$ with respect to~$q$, i.e. the map~$h$ with ~$qh\cong
s'$. Indeed, we have
\[
p_M^3h=p'qh\cong p's'=1_M
\]
and so~$h$ is a homotopy section of~$p_3^M$. Since the latter is a
Serre fibration, the homotopy lifting property yields an actual
section.

Let~$F_1$ and~$F_2$ be the fibers of fibrations~$p_3^M$ and~$p'$,
respectively. Consider the commutative diagram generated by the homotopy exact
sequences
of the Serre fibrations~$p_3^M$ and~$p'$:
\[
\CD
\pi_i(F_1) @>>> \pi_i(G_3(M)) @>(p_3^M)_*>> \pi_i(M) @>>>
\pi_{i-1}(F_1) @>>>\cdots\\
 @VV\phi_*V @ VVq_*V @VV=V @VV\phi_*V @.\\
\pi_i(F_2) @>>> \pi_i(Z) @>(p')_*>> \pi_i(M) @>>>\pi_{i-1}(F_2)@>>>
\cdots .\\
\endCD
\]
Note that we have
\[
\phi=\Omega (f)\ast \Omega(f)\ast\Omega(f)\ast \Omega(f).
\]

By \propref{p:join} and since~$\pi_i(\Omega f)=0$ for~$i\le n-3$, we
conclude that~$\pi_i(\phi)=0$ for~$i\le n-3+3=n$.  Hence~$\phi$
induces an isomorphism of the homotopy groups of dimensions~$\le n-1$
and an epimorphism in dimension~$n$.  By the Five Lemma we obtain
that~$q_*$ is an isomorphism in dimensions~$\le n-1$ and an
epimorphism in dimension~$n$. Hence the homotopy fiber of~$q$
is~$(n-1)$-connected. Since~$\dim M=n$, the map~$s'$ admits a homotopy
lifting~$h:M\to G_3(M)$.
\end{proof}

\begin{cor}
\label{c:real}
Given a finitely presented group~$\pi$ and non-negative integer
numbers~$k, l$ there exists a closed manifold~$M$ such
that~$\pi_1(M)=\pi$, while~$\cat M=3+k$ and~$\dim M=5+2k+l$.
\end{cor}

\begin{proof}
By \theoref{t:high}, there exists a manifold~$N$ such that
$\pi_1(M)=\pi$,~$\cat M=3$ and~$\dim M=5+l$. Moreover, this manifold
$N$ possesses a detecting element, i.e. a cohomology class whose
category weight is equal to~$\cat N=3$. For~$\pi$ free this follows
since the cuplength of~$N$ is equal to 3, for other groups we have the
detecting element~$f^*u\cup v$ constructed in the proof of
\theoref{th:main}. If a space~$X$ possesses a detecting element then,
for every~$k>0$, we have~$\cat (X\times S^k) =\cat X+1$ and~$X \times
S$ possesses a detecting element, \cite{R2}. Now, the manifold
$M:=N\times (S^2)^k$ is the desired manifold.
\end{proof}

Generally, we have a question about relations between the category,
the dimension, and the fundamental group of a closed manifold. The
following proposition shows that the situation quite intricate.

\begin{prop}
\label{p:lense}
Let $p$ be an odd prime. Then there exists a closed $(2n+1)$-manifold
with $\cat M=\dim M$ and $\pi_1(M)=\Zp$, but there are no closed
$2n$-manifolds with $\cat M=\dim M$ and $\pi_1(M)=\Zp$.
\end{prop}

\begin{proof}
An example of $(2n+1)$-manifold is the quotient space $S^{2n+1}/\Zp$
with respect to a free $\Zp$-action on $S^{2n+1}$. Now, given a
$2n$-manifold with $\pi_1(M)=\Zp$, consider a map $f:M \to K(\Zp,1)$
that induces an isomorphism of fundamental groups. Since
$H_{2n}(K(\Zp,1))=0$, it follows from the obstruction theory and
Poincar\'e duality that $f$ can be deformed into the $(2n-1)$-skeleton
of $K(\Zp,1)$, cf.~\cite[Section 8]{Bab1}. Hence, $M$ is not
$2n$-essential, and thus $\cat M <2n$ \cite{KR1}.
\end{proof}

%\vfill\eject

\end{document}